\documentclass[preprint,12pt,3p]{elsarticle}
\journal{arXiv}

\usepackage[english]{babel}
\usepackage[utf8]{inputenc}
\usepackage[T1]{fontenc}
\usepackage{amsfonts,amsmath,amssymb,amsthm,tikz}
\usepackage{patternmacros}
\usepackage[all,cmtip]{xy}
\usepackage{graphicx}
\usetikzlibrary{shapes,patterns}
\usepackage{color,subcaption}
\usepackage{mathtools}
\usepackage{url}

\newtheorem{theorem}{Theorem}[section]
\newtheorem{corollary}[theorem]{Corollary}

\newtheorem{lemma}[theorem]{Lemma}
\newtheorem{conjecture}[theorem]{Conjecture}

\newtheorem{example}[theorem]{Example}
\newtheorem{question}[theorem]{Question}

\begin{document}

\begin{frontmatter}

\title{Asymptotic Behaviour of the Containment of Certain Mesh Patterns}

\author[a,b]{Dejan Govc}					
\author[a,d]{Jason P. Smith}
\address[a]{Institute of Mathematics, University of Aberdeen, Aberdeen, UK}
\address[b]{Faculty of Mathematics and Physics, University of Ljubljana, Ljubljana, Slovenia}
\address[d]{Department of Mathematics and Physics, Nottingham Trent University, Nottingham, UK}
\fntext[fn1]{Both authors acknowledge support from EPSRC grant EP/P025072/. The first named author was also supported in part by the Slovenian Research Agency programme P1-0292 and grant N1-0083. The second named author was also supported in part by \'Ecole Polytechnique F\'ed\'erale de Lausanne.}

\begin{abstract}
We present some results on the proportion of permutations of length $n$
 containing certain mesh patterns as $n$ grows large, and give exact 
 enumeration results in some cases. In particular, we focus on mesh 
 patterns where entire rows and columns are shaded. We prove some general 
 results which apply to mesh patterns of any length, and then consider mesh patterns
 of length four. An important consequence of these results is to show that the 
 proportion of permutations containing a mesh pattern can take a wide range of values
 between $0$ and $1$.
\end{abstract}

\begin{keyword}
mesh patterns, enumerative combinatorics, permutation patterns
\end{keyword}

\end{frontmatter}

\section{Introduction}

Mesh patterns are a generalisation of permutations patterns, and were first introduced
 by Br\"and\'en and Claesson in \cite{Bra11}.
 A mesh pattern consists of a pair $(\pi,P)$, where $\pi$ is a permutation and $P$
 is a set of coordinates in a square grid. For example, $(312,\{(0,0),(1,2)\})$ is a
 mesh pattern, which we depict by
 \begin{center}
 \patt{0.5}{3}{3,1,2}[0/0,1/2][][][][][4].
 \end{center}

Permutations patterns have long been a topic of much interest, primarily due to their
 links to sorting algorithms, see \cite{Kit11} for an excellent overview of the field.
 Of particular interest in permutation patterns is the enumeration of the avoidance or containment class
 of particular permutations.
 Mesh patterns have been studied extensively since their introduction,
 see e.g.,~\cite{AKV13,Ten13,CTU15,JKR15,BGMU19}.
 The first systematic study of the avoidance of mesh patterns was conducted in \cite{HJSVU15},
 where enumeration results were given for the avoidance of $25$ patterns of length $2$.
 The first study of the distribution of the avoidance of mesh patterns was undertaken in \cite{KZ19}, 
 which was further extended in~\cite{KZZ20}.

In this paper we present some results on the proportion of permutations
 containing certain mesh patterns as $n$ grows large, that is, the asymptotic
 behaviour of
 $s_n^+(p)/n!$ where~$p$ is a mesh pattern and $s_n^+(p)$ is the number of 
 permutations of length $n$ containing at least one occurrence of $p$.
 In particular, does the limit in Expression~\eqref{eq:1}, which we call the \emph{containment limit of p}, exist and can we compute it?
 \begin{equation}\lim_{n\to\infty}\frac{s_n^+(p)}{n!}\label{eq:1}\end{equation}

It is known that in traditional permutation patterns as $n$ grows large the
 containment limit tends to $1$, which gives us our first result for mesh patterns. 
 For any permutation $\pi$ we get the containment limit
 $$\lim_{n\to\infty}\frac{s_n^+((\pi,\{\}))}{n!}=1.$$ 
 Moreover, if $p$ is a mesh pattern where every box is shaded, 
 then it is not possible for a larger permutation to contain $p$, 
 which gives us our second result
 $$\lim_{n\to\infty}\frac{s_n^+(p)}{n!}=0.$$ 
 It is often the case when studying permutation patterns that the asymptotics 
 of such proportions either tend to $0$ or $1$. So the motivation for this article is to find 
 mesh patterns where the containment limit lies strictly between $0$ and~$1$.

In Section~\ref{sec:2} we recall some definitions and notation.
 In Section~\ref{sec:3} we present some formulas for the containment limit of 
 mesh patterns $(\pi,R)$ for any permutation~$\pi$, and a fixed type of shading~$R$.
 In Section~\ref{sec:4} we present some formulas for the containment limit for 
 mesh patterns $(\pi,R)$, where~$\pi$ is a permutation of length four and~$R$ is fixed.
 We finish with some conjectures and further questions in Section~\ref{sec:5}.


\section{Definitions and Notation}\label{sec:2}
First we recall some definitions concerning permutation patterns.
 Let $[n]:=\{0,1,\dots,n\}$ and $[m,n]:=\{m,m+1,\dots,n\}$. We consider a permutation $\pi$ using one line notation, so as a
 sequence of the numbers in $[1,n]$ and we say the length of $\pi$, denoted $|\pi|$, is $n$.
 Let $S_n$ be the set of permutations of length $n$.
 We denote the number in the $i$-th position of $\pi$ by $\pi_i$.
 A permutation $\sigma$ \emph{occurs} in a permutation~$\pi$ if there is a subsequence,
 $\eta$, of $\pi$ whose letters appear in the same relative order of size as the letters
 of $\sigma$. The subsequence $\eta$ is called an \emph{occurrence} of~$\sigma$ in 
 $\pi$. If no such occurrence exists we say that $\pi$ \emph{avoids} $\sigma$.
 For example, $213$ occurs as a pattern in $3124$ as the subsequence $314$, 
 but $3124$ avoids $321$.

A \emph{mesh pattern} is a pair $(\pi,R)$, where $\pi$ is a permutation of length~$n$
 and $R\subseteq[n]\times[n]$. We depict a mesh pattern on a grid by putting dots in positions $(i,\pi_i)$, for all $i\in[n]$,
 and for each coordinate $(x,y)\in R$ shade the boxes whose south west corner is $(x,y)$.
 The \emph{length} of a mesh pattern is given by $|p|=|\pi|$.
 For example, $(3124,\{(0,2),(1,2),(3,3)\})$ is depicted by
 \begin{center}
 \patt{0.3}{4}{3,1,2,4}[0/2,1/2,3/3][][][][][4].
 \end{center}

Consider a mesh pattern $(\sigma,S)$ and an occurrence $\eta$ of $\sigma$ in $\pi$,
 in the classical permutation pattern sense.
 If $(i,j)$ is a dot in the plot of $\sigma$, let $(\alpha_\eta(i),\beta_\eta(j))$ be the corresponding dot in~$\pi$ given by $\eta$. 
 Each box $(i,j)$ of $S$ corresponds to a rectangular
 area in $\pi$ consisting of the boxes $$R_{\eta}(i,j)=[\alpha_\eta(i),\alpha_\eta(i+1)-1]\times[\beta_\eta(j),\beta_\eta(j+1)-1],$$
 where $\alpha_\eta(0)=\beta_\eta(0)=0$ and $\alpha_\eta(|\sigma|+1)=\beta_\eta(|\sigma|+1)=|\pi|+1$.
 For example, in Figure~\ref{fig:occEx} where $\eta$ is the occurrence in red, the area
 corresponding to the box $(1,1)$ is $$R_\eta(1,1)=\{(1,1),(2,1),(1,2),(2,2)\}.$$
 A point is contained in $R_\eta(i,j)$ if it is in the
 interior of $R_\eta(i,j)$, that is, not on the boundary.
 We say that $\eta$ is an \emph{occurrence} of the mesh pattern 
 $(\sigma,S)$ in the permutation~$\pi$ if there is no point in $R_{\eta}(i,j)$, 
 for all shaded boxes~$(i,j)\in S$.

\begin{figure}[h]
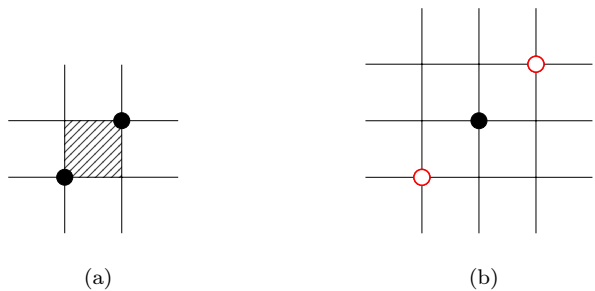
\centering
 \begin{subfigure}[b]{0.3\textwidth}
  \centering\patt{0.75}{2}{1,2}[1/1][][][][]
  \caption{}\label{subfiga}\end{subfigure}
 \begin{subfigure}[b]{0.3\textwidth}\centering
  \colpatt{0.75}{3}{1,2,3}[][1/1,3/3][]
  \caption{}\label{subfigc}\end{subfigure}
  \caption{A mesh pattern (a) and a permutation (b), where the hollow red points in~(b)
   indicate the only pair which is not an occurrence of (a) in (b).
  }\label{fig:occEx}
\end{figure}

\section{Patterns of general length}\label{sec:3}
In this section we consider results that relate to mesh patterns of any length.
 We begin with a useful lemma that allows us to construct bounds for the containment limit 
 of $p$ using other mesh patterns.
 
 \begin{lemma}\label{lem:subshade}
Consider two mesh patterns $p_1=(\pi,R_1)$ and $p_2=(\pi,R_2)$. If~$R_1\subseteq R_2$, then 
 \[
 s_n^+(p_1)\ge s_n^+(p_2).
 \]
\begin{proof}
Suppose $p_2$ occurs in $\tau$ as the occurrence $\eta$. 
 So $\eta$ is an occurrence of~$\pi$ in the classical sense, and since every shaded block $(i,j)$ of $p_1$ is also shaded
 in $p_2$ we know that there are no points in the areas $R_\eta(i,j)$, for all $(i,j)\in R_1$.
 Hence,~$\eta$ is also an occurrence of $p_1$. So every permutation that contains an occurrence of
 $p_2$ also contains an occurrence of $p_1$, which implies $s_n^+(p_1)\ge s_n^+(p_2)$.
\end{proof}
\end{lemma}

Next we consider the case where all boxes are shaded except one row (or column) which is fully unshaded, such as the mesh pattern in Figure~\ref{subfig:1a}.
\begin{figure}[b]
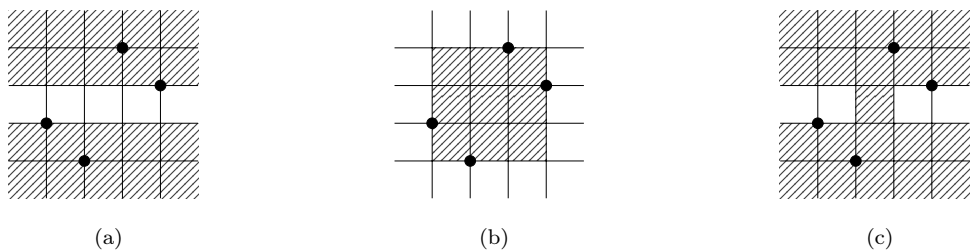
\centering
  \begin{subfigure}[b]{0.3\textwidth}
    \centering
    \patt{0.5}{4}{2,1,4,3}[0/0,1/0,2/0,3/0,4/0,0/1,1/1,2/1,3/1,4/1,0/3,1/3,2/3,3/3,4/3,0/4,1/4,2/4,3/4,4/4]
    \caption{}\label{subfig:1a}
  \end{subfigure}
  \begin{subfigure}[b]{0.3\textwidth}
    \centering
    \patt{0.5}{4}{2,1,4,3}[1/1,2/1,3/1,1/2,2/2,3/2,1/3,2/3,3/3]
    \caption{}\label{subfig:1b}
  \end{subfigure}
  \begin{subfigure}[b]{0.3\textwidth}
    \centering
    \patt{0.5}{4}{2,1,4,3}[0/0,1/0,2/0,3/0,4/0,0/1,1/1,2/1,3/1,4/1,2/2,0/3,1/3,2/3,3/3,4/3,0/4,1/4,2/4,3/4,4/4]
    \caption{}\label{subfig:1c}
  \end{subfigure}
  \caption{Examples of the mesh patterns considered in Theorem~\ref{thm:row}, Corollary~\ref{cor:boxed} and Theorem~\ref{thm:row1}} \label{fig:1}
\end{figure}

\begin{theorem}\label{thm:row}
Let $p=(\pi,R)$ be a mesh pattern with $|\pi|=k$ and $i\in[k]$, where $R=[k]\times([k]\setminus\{i\})$,
 that is, we fully shade all rows except row $i$ which is fully unshaded, then for $n\geq k$ we have
 \[
 s_n^+(p)=\frac{n!}{k!}\,\,\,\,\,\,\,\,\,\,\,\,\text{ and }\,\,\,\,\,\,\,\,\,\,\,\,\frac{s_n^+(p)}{n!}=\frac1{k!}.
 \]
\end{theorem}

\begin{proof}
Let $0\leq i\leq k$ and suppose the $i$-th row is the one that is not shaded. 
 Then the shading requires that the subword of $\sigma$ realising the mesh pattern 
 consists of the letters ${[1,i]\cup[n-k+i+1,n]}$. Every permutations contains 
 these letters, so for the mesh pattern to occur, they just need to be permuted correctly. 
 This immediately implies the claim.
\end{proof}

\begin{corollary}
Let $p=(\pi,R)$ be a mesh pattern with $|\pi|=k$ and $i\in[k]$, where $R=([k]\setminus\{j\})\times[k]$,
 that is, we fully shade all columns except column $j$ which is fully unshaded, then for $n\geq k$ we have
 \[
 s_n^+(p)=\frac{n!}{k!}\,\,\,\,\,\,\,\,\,\,\,\,\text{ and }\,\,\,\,\,\,\,\,\,\,\,\,\frac{s_n^+(p)}{n!}=\frac1{k!}.
 \]
\begin{proof}
This follows from Theorem~\ref{thm:row} by rotational symmetry, that is, if $\tau$ contains $p$ then~$\hat{\tau}$ 
 contains $\hat{p}$, where $\hat{\tau}$ and $\hat{p}$ are obtained by rotating $\tau$ and $p$ by $90$ degrees.
\end{proof}
\end{corollary}

Combining Lemma~\ref{lem:subshade} and Theorem~\ref{thm:row} gives us a lower bound for the containment limit of boxed patterns.
 \emph{Boxed patterns} are mesh patterns where everything is shaded except for the first and last row and the first and last column
 all of which are completely unshaded, such as the mesh pattern in Figure~\ref{subfig:1b}. Boxed patterns were extensively studied in \cite{AKV13}.
\begin{corollary}\label{cor:boxed}
Let $p=(\pi,R)$ be a boxed pattern, so $R=[1,|\pi|-1]\times[1,|\pi|-1]$, then
 \[
 \lim_{n\to\infty}\frac{s_n^+(p)}{n!}\geq\frac1{k!}.
 \]
\end{corollary}

In Theorem~\ref{thm:row} the containment limit is nonzero. We now show that shading one additional square, 
 for example the mesh pattern in Figure~\ref{subfig:1c}, reduces this limit to $0$.

\begin{theorem}\label{thm:row1}
Let $p=(\pi,R)$ be a mesh pattern with $|\pi|=k$ and $i,j\in[k]$, where ${R=[k]\times([k]\setminus\{i\})\cup\{(j,i)\}}$,
 that is, we fully shade all rows except row~$i$ which has exactly one shaded box, then for $n\geq k$ we have
 \[
 s_n^+(p)=\frac{(n-1)!}{(k-1)!}
 \,\,\,\,\,\,\,\,\,\,\,\,\text{ and }\,\,\,\,\,\,\,\,\,\,\,\,
 \lim_{n\to\infty}\frac{s_n^+(p)}{n!}=0.
 \]
\end{theorem}

\begin{proof}
Let $a=\pi_j$ when $j>0$ and $a=\pi_1$ when $j=0$. The position of $a$ in any occurrence of $p$ is uniquely determined, since if $j=0$, 
 the letter corresponding to $a$ must occur in the first place; if $j=k$, the letter corresponding to $a$ 
 must occur in the last place; and if~$1\leq j\leq k-1$, the letter corresponding to $a$ must appear immediately before $\pi_{j+1}$.
 This implies that $s_n^+(p)=s_n^+(\hat{p})$, where $\hat{p}$ is the 
 mesh pattern obtained from $p$ by deleting $a$ and unshading $(j,i)$. 
 So $\hat{p}$ is a mesh pattern of length $n-1$ with all rows shaded except one which is fully unshaded, and the result follows from Theorem~\ref{thm:row}.
\end{proof}


\section{Patterns of length four}\label{sec:4}

In this section we consider mesh patterns of length four, with a particular focus on the permutation $2143$ and its symmetries.
 We begin by considering the mesh pattern in Figure~\ref{fig:border}, where the shaded boxes are exactly the boundary boxes. 

\begin{figure}[b]
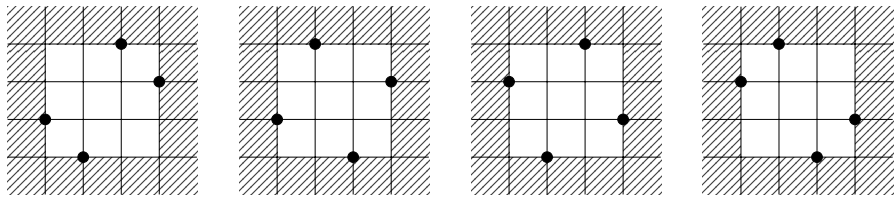
\centering
 \patt{0.5}{4}{2,1,4,3}[0/0,1/0,2/0,3/0,4/0,0/1,4/1,0/2,4/2,0/3,4/3,0/4,1/4,2/4,3/4,4/4]
 \patt{0.5}{4}{2,4,1,3}[0/0,1/0,2/0,3/0,4/0,0/1,4/1,0/2,4/2,0/3,4/3,0/4,1/4,2/4,3/4,4/4]
 \patt{0.5}{4}{3,1,4,2}[0/0,1/0,2/0,3/0,4/0,0/1,4/1,0/2,4/2,0/3,4/3,0/4,1/4,2/4,3/4,4/4]
 \patt{0.5}{4}{3,4,1,2}[0/0,1/0,2/0,3/0,4/0,0/1,4/1,0/2,4/2,0/3,4/3,0/4,1/4,2/4,3/4,4/4]
 \caption{The mesh patterns considered in Theorem~\ref{thm:border}}\label{fig:border}
\end{figure}

\begin{theorem}\label{thm:border}
Consider the mesh pattern $p=(\pi,(\{0,4\}\times[4])\cup([4]\times\{0,4\})$, where ${\pi\in\{2143,2413,3142,3412\}}$,
 that is, any of the mesh patterns in Figure~\ref{fig:border}.
 Then for $n\ge4$,
 \[
 s_n^+(p) = \binom{n-2}{2}^2(n-4)!
 \qquad\text{and}\qquad
 \lim_{n\to\infty}\frac{s_n^+(p)}{n!}=\frac14.
 \]
\end{theorem}

\begin{proof}
Assume $\pi=2143$, the argument for the other permutations is analogous. 
 Note that if $p$ occurs in a permutation $\sigma\in S_n$, the shading tells us that there can be nothing to the left of the $2$,
 nothing to the right of the $3$, nothing above $4$ and nothing below $1$.
 This means that $p$ can be realised in $\sigma$ in exactly one way, namely as the subword $\sigma_1 1 n \sigma_n$, subject to the restriction $1<\sigma_1<\sigma_n<n$.
 Conversely, a subword of this form is an occurrence of the pattern $p$.
 Therefore, to calculate $s_n^+(p)$, it suffices to count how many permutations contain this subword.

In a permutation $\sigma$ of length $n$, there are $n-2$ possible places for $1$ and $n$.
 Since $1$ has to occur before $n$, we can only choose the two places, leaving us with $\binom{n-2}{2}$ choices.
 Similarly, there are $n-2$ possible values of $\sigma_1$ and $\sigma_n$ and since we must have $\sigma_1<\sigma_n$, 
 this again gives us $\binom{n-2}{2}$ possibilities. Since there are $(n-4)$ remaining letters which we can freely permute, we conclude
 \[
 s_n^+(p)=\binom{n-2}{2}^2(n-4)!
 \]
 Therefore,
 \[
 \lim_{n\to\infty}\frac{\binom{n-2}{2}^2(n-4)!}{n!}=\lim_{n\to\infty}\frac{(n-2)(n-3)}{4n(n-1)} = \frac14.
 \]
\end{proof}

Considering the remaining $20$ permutations in $S_4$ with the same shading leads to a different result:

\begin{theorem}\label{thm:border0}
Consider the mesh pattern $p=(\pi,(\{0,4\}\times[4])\cup([4]\times\{0,4\})$, where $\pi\in S_4\setminus\{2143,2413,3142,3412\}$,
 that is, with the same shading as in Figure~\ref{fig:border}, but a different choice of permutation.
 Then,
 \[
 \lim_{n\to\infty}\frac{s_n^+(p)}{n!}=0.
 \]
\end{theorem}

\begin{proof}
For each such $\pi$, we either have $\pi_1\in\{1,4\}$ or $\pi_4\in\{1,4\}$. Therefore, if $\sigma\in S_n$
 has an occurrence of $(\pi,R)$, the shading then prescribes either an explicit value of $\sigma_1\in\{1,n\}$ 
 or of $\sigma_n\in\{1,n\}$. But fixing one letter in $\sigma$ already reduces the number of such permutations 
 to $(n-1)!$, which is vanishingly small in $S_n$ as $n\to\infty$.
\end{proof}

The exact enumeration for the mesh patterns in Theorem~\ref{thm:border0} depends on the pattern chosen and we do not pursue it further here.

\begin{figure}[b]
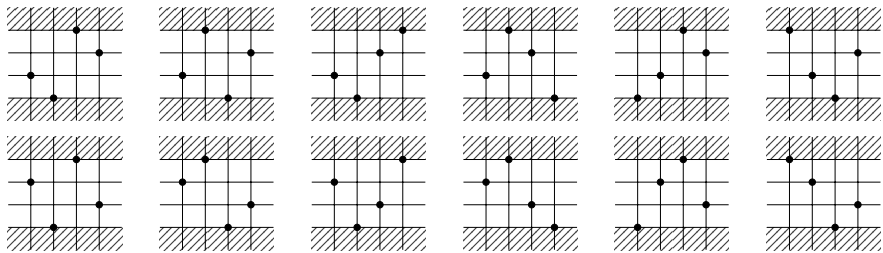

 \begin{center}
  \patt{0.30}{4}{2,1,4,3}[0/0,1/0,2/0,3/0,4/0,0/4,1/4,2/4,3/4,4/4]
  \patt{0.30}{4}{2,4,1,3}[0/0,1/0,2/0,3/0,4/0,0/4,1/4,2/4,3/4,4/4]
  \patt{0.30}{4}{2,1,3,4}[0/0,1/0,2/0,3/0,4/0,0/4,1/4,2/4,3/4,4/4]
  \patt{0.30}{4}{2,4,3,1}[0/0,1/0,2/0,3/0,4/0,0/4,1/4,2/4,3/4,4/4]
  \patt{0.30}{4}{1,2,4,3}[0/0,1/0,2/0,3/0,4/0,0/4,1/4,2/4,3/4,4/4]
  \patt{0.30}{4}{4,2,1,3}[0/0,1/0,2/0,3/0,4/0,0/4,1/4,2/4,3/4,4/4]\linebreak
  \patt{0.30}{4}{3,1,4,2}[0/0,1/0,2/0,3/0,4/0,0/4,1/4,2/4,3/4,4/4]
  \patt{0.30}{4}{3,4,1,2}[0/0,1/0,2/0,3/0,4/0,0/4,1/4,2/4,3/4,4/4]
  \patt{0.30}{4}{3,1,2,4}[0/0,1/0,2/0,3/0,4/0,0/4,1/4,2/4,3/4,4/4]
  \patt{0.30}{4}{3,4,2,1}[0/0,1/0,2/0,3/0,4/0,0/4,1/4,2/4,3/4,4/4]
  \patt{0.30}{4}{1,3,4,2}[0/0,1/0,2/0,3/0,4/0,0/4,1/4,2/4,3/4,4/4]
  \patt{0.30}{4}{4,3,1,2}[0/0,1/0,2/0,3/0,4/0,0/4,1/4,2/4,3/4,4/4]
 \end{center}
 \caption{The mesh patterns considered in the first case of Theorem~\ref{topbottom}, which consists of the length 4 permutations where 2 and 3 are not adjacent.}\label{fig:topbottom1}
\end{figure}
Next we consider mesh patterns where the shaded boxes are exactly the top and bottom rows 
 and $\pi$ is any length four permutation. The exact enumeration result depends on whether 
 $2$ and $3$ are consecutive or not, so there are two cases to treat, see Figures~\ref{fig:topbottom1} 
 and \ref{fig:topbottom2}. The limit, however, is the same in both cases.

\begin{figure}[t]
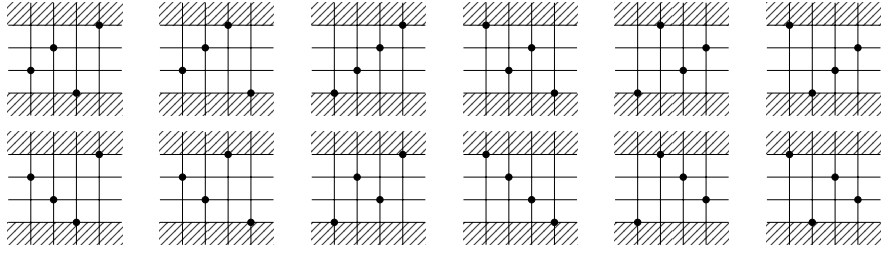

 \begin{center}
  \patt{0.30}{4}{2,3,1,4}[0/0,1/0,2/0,3/0,4/0,0/4,1/4,2/4,3/4,4/4]
  \patt{0.30}{4}{2,3,4,1}[0/0,1/0,2/0,3/0,4/0,0/4,1/4,2/4,3/4,4/4]
  \patt{0.30}{4}{1,2,3,4}[0/0,1/0,2/0,3/0,4/0,0/4,1/4,2/4,3/4,4/4]
  \patt{0.30}{4}{4,2,3,1}[0/0,1/0,2/0,3/0,4/0,0/4,1/4,2/4,3/4,4/4]
  \patt{0.30}{4}{1,4,2,3}[0/0,1/0,2/0,3/0,4/0,0/4,1/4,2/4,3/4,4/4]
  \patt{0.30}{4}{4,1,2,3}[0/0,1/0,2/0,3/0,4/0,0/4,1/4,2/4,3/4,4/4]\linebreak
  \patt{0.30}{4}{3,2,1,4}[0/0,1/0,2/0,3/0,4/0,0/4,1/4,2/4,3/4,4/4]
  \patt{0.30}{4}{3,2,4,1}[0/0,1/0,2/0,3/0,4/0,0/4,1/4,2/4,3/4,4/4]
  \patt{0.30}{4}{1,3,2,4}[0/0,1/0,2/0,3/0,4/0,0/4,1/4,2/4,3/4,4/4]
  \patt{0.30}{4}{4,3,2,1}[0/0,1/0,2/0,3/0,4/0,0/4,1/4,2/4,3/4,4/4]
  \patt{0.30}{4}{1,4,3,2}[0/0,1/0,2/0,3/0,4/0,0/4,1/4,2/4,3/4,4/4]
  \patt{0.30}{4}{4,1,3,2}[0/0,1/0,2/0,3/0,4/0,0/4,1/4,2/4,3/4,4/4]
 \end{center}
 \caption{The mesh patterns considered in the second case of Theorem~\ref{topbottom}, which consists of the length 4 permutations where 2 and 3 are adjacent.}\label{fig:topbottom2}
\end{figure}

\begin{theorem}\label{topbottom}
Let $p=(\pi,R)$, where $\pi\in S_4$ and $R=([4]\times0)\cup([4]\times4)$, 
 that is, the shaded cells are exactly those on the top and bottom rows.
 \begin{enumerate}
  \item If $2$ and $3$ do not appear consecutively in $\pi$,
   so $p$ is any of the mesh patterns in Figure~\ref{fig:topbottom1}, then
   \[
   s_n^+(p) = (n-2)!\sum_{k=2}^{n-2}\sum_{\ell=k+1}^{n-1}\left(1-\binom{n-\ell+k-1}{k-1}^{-1}\right).
   \]
  \item If $2$ and $3$ do appear consecutively in $\pi$, 
   so $p$ is any of the mesh patterns in Figure~\ref{fig:topbottom2}, then
   \[
   s_n^+(p) = (n-2)!\sum_{k=3}^{n-1} (n-k)\left(1-\frac1{(k-1)!}\right).
   \]
\end{enumerate}
In both cases,
 \[
 \lim_{n\to\infty}\frac{s_n^+(p)}{n!} = \frac12.
 \]
\end{theorem}

\begin{proof}
First we derive the two enumerative results. The proof proceeds by splitting the permutation $\sigma\in S_n$
 into three segments separated by the letters that realise $1$ and $4$ of the pattern.
 The two cases then arise by considering whether or not $2$ and $3$ lie in the same segment.

{\bf Case 1:} $2$ and $3$ are not adjacent. Suppose $\pi=2143$, the argument for the other permutations is analogous.
 Note that if $p$ occurs in a permutation $\sigma\in S_n$, the shading tells us that there can be nothing above $4$ and nothing below $1$. 
 This means that a realisation of~$p$ in $\sigma$ is exactly a subword $a1nb$, subject to the restriction $1<a<b<n$. 

First we fix $\sigma_k=1$ and $\sigma_\ell=n$, with $k<\ell$. Next we count the permutations which fail to satisfy the condition that
 a letter before the letter $1$ is smaller than a letter after the letter $n$ in $\sigma$.
 To do so we choose the letters which lie between $1$ and $n$, for which we have~$\binom{n-2}{\ell-k-1}(\ell-k-1)!$ possibilities
 as there are $\ell-k-1$ letters to pick out of $n-2$, and we can order them in any way. 
 The remaining letters must then be partitioned so that the largest $k-1$ are before $1$ in $\sigma$, and the smallest $n-\ell$ are after $n$.
 We can order both these sets in any way. So the number of permutations which don't satisfy the condition 
 is $$\binom{n-2}{\ell-k-1}(\ell-k-1)!(k-1)!(n-\ell)!=\frac{(n-2)!}{\binom{n-\ell+k-1}{k-1}}.$$
 Therefore, we have $(n-2)!-\frac{(n-2)!}{\binom{n-l+k-1}{k-1}}$ permutations that do satisfy the condition.
 We then sum over all $k$ and $\ell$ to get the first enumerative result:
 \[
 s_n^+(p)=(n-2)!\sum_{k=2}^{n-2}\sum_{\ell=k+1}^{n-1}\left(1-\binom{n-\ell+k-1}{k-1}^{-1}\right).
 \]

{\bf Case 2:} $2$ and $3$ are adjacent. Suppose $\pi=2314$, the argument for the other permutations is analogous.
 Again, if $p$ occurs in a permutation $\sigma\in S_n$, the shading tells us that there can be nothing above $4$ and nothing below $1$. 
 This means that a realisation of $p$ in $\sigma$ is exactly a subword $ab1n$, subject to the restriction $1<a<b<n$. 

Fix $\sigma_k=1$ and $\sigma_\ell=n$, with $k<\ell$. Next we count the permutations which fail to satisfy the condition that
 two of the letters appearing before $1$ in $\sigma$ occur in ascending order as required to realise the pattern.
 First we choose the letters appearing after $1$ (excluding $n$ which has already been fixed), for which we have $\binom{n-2}{n-k-1}(n-k-1)!$ possibilities
 as there are $n-k-1$ letters to pick out of $n-2$, and we can order them in any way. 
 The remaining letters are then the ones occurring before $1$ in $\sigma$ and for the condition to fail, they all
 have to appear in decreasing order. Once the letters after $1$ have been fixed, there is exactly one way to do this,
 so the number of permutations which don't satisfy the condition 
 is $$\binom{n-2}{n-k-1}(n-k-1)!=\frac{(n-2)!}{(k-1)!}.$$
 Therefore, we have $(n-2)!-\frac{(n-2)!}{(k-1)!}$ permutations that do satisfy the condition.
 We then sum over all $k$ and $\ell$ to get the second enumerative result:
 \[
 s_n^+(p)=(n-2)!\sum_{k=3}^{n-1}\sum_{\ell=k+1}^{n}\left(1-\frac{1}{(k-1)!}\right)=(n-2)!\sum_{k=3}^{n-1}(n-k)\left(1-\frac{1}{(k-1)!}\right).
 \]
 This concludes the proof of the two enumerative results.

Finally, we prove the asymptotic result. In both cases we have the exact same lower and upper bounds,
 but we obtain them in slightly different ways. In the first case, the upper bound is obtained as follows:
 \[
 \frac{s_n^+(p)}{(n-2)!}=\sum_{k=2}^{n-2}\sum_{\ell=k+1}^{n-1}\left(1-\binom{n-\ell+k-1}{k-1}^{-1}\right)\leq\sum_{k=2}^{n-2}\sum_{\ell=k+1}^{n-1}1=\binom{n-2}{2}.
 \]
 In the second case, we instead have:
 \[
 \frac{s_n^+(p)}{(n-2)!}=\sum_{k=3}^{n-1}(n-k)\left(1-\frac{1}{(k-1)!}\right)\leq\sum_{k=3}^{n-1}(n-k)=\binom{n-2}{2}.
 \]
 In both cases, this implies:
 \[
 \lim_{n\to\infty}\frac{s_n^+(p)}{n!}\le\lim_{n\to\infty}\frac{(n-2)!}{n!}\binom{n-2}{2}=\lim_{n\to\infty}\frac{(n-2)(n-3)}{2n(n-1)}=\frac12.
 \]
 To get the lower bound in the first case, first note that $1-\binom{n-\ell_1+k-1}{k-1}^{-1}\ge 1-\binom{n-\ell_2+k-1}{k-1}^{-1}$ if and only if $\ell_1\le\ell_2$.
 So if $\ell\le n-1$ we get $$1-\binom{n-\ell+k-1}{k-1}^{-1}\ge1-\binom{k}{k-1}^{-1}=\frac{k-1}{k}.$$ Which gives the following inequality
 \begin{align*}
  \frac{s_n^+(p)}{n!}&\ge
  \frac{(n-2)!}{n!}\sum_{k=2}^{n-2}(n-k-1)\frac{k-1}{k}\\
 \end{align*}
 In the second case, note that $1-\frac{1}{(k-1)!}\ge 1-\frac1{k-1}$ holds for all $k\geq 2$, and then change the summation index to $r=k-1$. This gives
 \begin{align*}
  \frac{s_n^+(p)}{n!}&\ge
  \frac{(n-2)!}{n!}\sum_{k=3}^{n-1}(n-k)\frac{k-2}{k-1}\\
  &=\frac{(n-2)!}{n!}\sum_{r=2}^{n-2}(n-r-1)\frac{r-1}{r}.
 \end{align*}
 Notice that we obtain the exact same lower bound in both cases. This bound can now be simplified as follows:
 \begin{align*}
  \frac{(n-2)!}{n!}\sum_{k=2}^{n-2}(n-k-1)\frac{k-1}{k}&=\frac{1}{n(n-1)}\sum_{k=2}^{n-2}\left[n-k-\frac{n-1}{k}\right]\\
  &=\frac{n(n-3)}{n(n-1)}-\frac{n(n-3)}{2n(n-1)}-\frac{(n-1)(H_{n-2}-1)}{n(n-1)},
 \end{align*}
 where $H_n$ is the $n$-th Harmonic number. Using the formula ${\displaystyle\lim_{n\to\infty}\left(H_n-\ln(n)\right)=\gamma}$ for Euler's constant \cite{Eul41},
 we get the following lower bound for the containment limit.
 \begin{align*}
  \lim_{n\to\infty}\frac{s_n^+(p)}{n!}\ge 1-\frac12-\lim_{n\to\infty}\frac{\gamma+\ln(n-2)}{n}=\frac12.
 \end{align*}
 Combining the two bounds implies the result.
\end{proof}

By trivial symmetries we get the following corollary.
\begin{corollary}\label{cor:topbottom0}
 Let $p$ be any of the following mesh patterns
 \begin{align*}
  &\patt{0.25}{4}{2,1,4,3}[0/0,0/1,0/2,0/3,0/4,4/0,4/1,4/2,4/3,4/4]
  \patt{0.25}{4}{2,4,1,3}[0/0,0/1,0/2,0/3,0/4,4/0,4/1,4/2,4/3,4/4]
  \patt{0.25}{4}{2,1,3,4}[0/0,0/1,0/2,0/3,0/4,4/0,4/1,4/2,4/3,4/4]
  \patt{0.25}{4}{2,4,3,1}[0/0,0/1,0/2,0/3,0/4,4/0,4/1,4/2,4/3,4/4]
  \patt{0.25}{4}{1,2,4,3}[0/0,0/1,0/2,0/3,0/4,4/0,4/1,4/2,4/3,4/4]
  \patt{0.25}{4}{4,2,1,3}[0/0,0/1,0/2,0/3,0/4,4/0,4/1,4/2,4/3,4/4]\\
  &\patt{0.25}{4}{3,1,4,2}[0/0,0/1,0/2,0/3,0/4,4/0,4/1,4/2,4/3,4/4]
  \patt{0.25}{4}{3,4,1,2}[0/0,0/1,0/2,0/3,0/4,4/0,4/1,4/2,4/3,4/4]
  \patt{0.25}{4}{3,1,2,4}[0/0,0/1,0/2,0/3,0/4,4/0,4/1,4/2,4/3,4/4]
  \patt{0.25}{4}{3,4,2,1}[0/0,0/1,0/2,0/3,0/4,4/0,4/1,4/2,4/3,4/4]
  \patt{0.25}{4}{1,3,4,2}[0/0,0/1,0/2,0/3,0/4,4/0,4/1,4/2,4/3,4/4]
  \patt{0.25}{4}{4,3,1,2}[0/0,0/1,0/2,0/3,0/4,4/0,4/1,4/2,4/3,4/4]
 \end{align*}
 Then,
 \[
 s_n^+(p) = (n-2)!\sum_{k=2}^{n-2}\sum_{\ell=k+1}^{n-1}\left(1-\binom{n-\ell+k-1}{k-1}^{-1}\right).
 \]
 Let $p$ be any of the following mesh patterns
 \begin{align*}
  &\patt{0.25}{4}{2,3,1,4}[0/0,0/1,0/2,0/3,0/4,4/0,4/1,4/2,4/3,4/4]
  \patt{0.25}{4}{2,3,4,1}[0/0,0/1,0/2,0/3,0/4,4/0,4/1,4/2,4/3,4/4]
  \patt{0.25}{4}{1,2,3,4}[0/0,0/1,0/2,0/3,0/4,4/0,4/1,4/2,4/3,4/4]
  \patt{0.25}{4}{4,2,3,1}[0/0,0/1,0/2,0/3,0/4,4/0,4/1,4/2,4/3,4/4]
  \patt{0.25}{4}{1,4,2,3}[0/0,0/1,0/2,0/3,0/4,4/0,4/1,4/2,4/3,4/4]
  \patt{0.25}{4}{4,1,2,3}[0/0,0/1,0/2,0/3,0/4,4/0,4/1,4/2,4/3,4/4]\\
  &\patt{0.25}{4}{3,2,1,4}[0/0,0/1,0/2,0/3,0/4,4/0,4/1,4/2,4/3,4/4]
  \patt{0.25}{4}{3,2,4,1}[0/0,0/1,0/2,0/3,0/4,4/0,4/1,4/2,4/3,4/4]
  \patt{0.25}{4}{1,3,2,4}[0/0,0/1,0/2,0/3,0/4,4/0,4/1,4/2,4/3,4/4]
  \patt{0.25}{4}{4,3,2,1}[0/0,0/1,0/2,0/3,0/4,4/0,4/1,4/2,4/3,4/4]
  \patt{0.25}{4}{1,4,3,2}[0/0,0/1,0/2,0/3,0/4,4/0,4/1,4/2,4/3,4/4]
  \patt{0.25}{4}{4,1,3,2}[0/0,0/1,0/2,0/3,0/4,4/0,4/1,4/2,4/3,4/4]
 \end{align*}
 Then,
 \[
 s_n^+(p) = (n-2)!\sum_{k=3}^{n-1} (n-k)\left(1-\frac1{(k-1)!}\right).
 \]
 Furthermore, in both cases:
 \[
 \lim_{n\to\infty}\frac{s_n^+(p)}{n!} = \frac12.
 \]
\end{corollary}

The Shading Lemma is a useful result introduced in \cite[Lemma 3.11]{HJSVU15} which gives conditions allowing
 extra boxes to be shaded which does not change the containment set. 
 Applying the Shading Lemma combined with Theorem~\ref{topbottom}, for $\pi=2143$, gives the following corollary.

\begin{corollary}\label{cor:topbottom1}
 If $p$ is any of the following mesh patterns
 \begin{align*}
  &\patt{0.25}{4}{2,1,4,3}[0/0,1/0,2/0,3/0,4/0,0/4,1/4,2/4,3/4,4/4,0/1]
  \patt{0.25}{4}{2,1,4,3}[0/0,1/0,2/0,3/0,4/0,0/4,1/4,2/4,3/4,4/4,0/2]
  \patt{0.25}{4}{2,1,4,3}[0/0,1/0,2/0,3/0,4/0,0/4,1/4,2/4,3/4,4/4,1/1]
  \patt{0.25}{4}{2,1,4,3}[0/0,1/0,2/0,3/0,4/0,0/4,1/4,2/4,3/4,4/4,1/2]
  \patt{0.25}{4}{2,1,4,3}[0/0,1/0,2/0,3/0,4/0,0/4,1/4,2/4,3/4,4/4,1/2,1/1]
  \patt{0.25}{4}{2,1,4,3}[0/0,1/0,2/0,3/0,4/0,0/4,1/4,2/4,3/4,4/4,0/1,1/1]
  \patt{0.25}{4}{2,1,4,3}[0/0,1/0,2/0,3/0,4/0,0/4,1/4,2/4,3/4,4/4,1/2,0/2]
  \patt{0.25}{4}{2,1,4,3}[0/0,1/0,2/0,3/0,4/0,0/4,1/4,2/4,3/4,4/4,0/1,0/2]\\
  &\patt{0.25}{4}{2,1,4,3}[0/0,1/0,2/0,3/0,4/0,0/4,1/4,2/4,3/4,4/4,4/2]
  \patt{0.25}{4}{2,1,4,3}[0/0,1/0,2/0,3/0,4/0,0/4,1/4,2/4,3/4,4/4,4/3]
  \patt{0.25}{4}{2,1,4,3}[0/0,1/0,2/0,3/0,4/0,0/4,1/4,2/4,3/4,4/4,3/3]
  \patt{0.25}{4}{2,1,4,3}[0/0,1/0,2/0,3/0,4/0,0/4,1/4,2/4,3/4,4/4,3/2]
  \patt{0.25}{4}{2,1,4,3}[0/0,1/0,2/0,3/0,4/0,0/4,1/4,2/4,3/4,4/4,3/3,3/2]
  \patt{0.25}{4}{2,1,4,3}[0/0,1/0,2/0,3/0,4/0,0/4,1/4,2/4,3/4,4/4,3/3,4/3]
  \patt{0.25}{4}{2,1,4,3}[0/0,1/0,2/0,3/0,4/0,0/4,1/4,2/4,3/4,4/4,3/2,4/2]
  \patt{0.25}{4}{2,1,4,3}[0/0,1/0,2/0,3/0,4/0,0/4,1/4,2/4,3/4,4/4,4/3,4/2]
 \end{align*}
 or obtained by the union of the shaded boxes from a mesh pattern on the top row and a mesh pattern on the bottom row.
 Then,
 \[
 s_n^+(p) = (n-2)!\sum_{k=2}^{n-2}\sum_{\ell=k+1}^{n-1}\left(1-\binom{n-\ell+k-1}{k-1}^{-1}\right)
 \quad\text{and}\quad
 \lim_{n\to\infty}\frac{s_n^+(p)}{n!} = \frac12.
 \]
 \begin{proof}
  This follows from Theorem~\ref{topbottom}, the Shading Lemma \cite[Lemma 3.11]{HJSVU15}, and the Simultaneous Shading Lemma \cite[Lemma 7.6]{CTU15}.
 \end{proof}
\end{corollary}

Note that we can derive a similar result to Corollary~\ref{cor:topbottom1} for the other mesh patterns 
 in Theorem~\ref{topbottom} and Corollary~\ref{cor:topbottom0}, but we omit them here for brevity.

We suspect that Theorem~\ref{topbottom} will extend to general mesh patterns of length $n$.
 Heuristically, shading the top and bottom row fixes the relative position of $1$ and $n$,
 with exactly half of the permutations having that relative position. 
 Containing the remaining $k-2$ points of the pattern means that the corresponding 
 permutation pattern occurs with its points appearing on the prescribed sides of $1$ and $n$. 
 As a consequence of the Marcus-Tardos Theorem~\cite{MT04}, we know that as $n$ increases all permutation patterns occur with probability~$1$, 
 so one would expect similar behaviour here. Which leads to the following conjecture.

\begin{conjecture}\label{conj:topbottom}
Let $p=(\pi,R)$ where $|\pi|=k$ and $R=([k]\times0)\cup([k]\times k)$, that is, the shaded cells are exactly those on the top and bottom rows, then
 \[
 \lim_{n\to\infty}\frac{s_n^+(p)}{n!} = \frac12.
 \]
\end{conjecture}

In Corollary~\ref{cor:topbottom1} we see that shading the top and bottom rows, 
 and all but one box on the rightmost column, gives a containment limit of $\frac{1}{2}$.
 Next we show that if we shade all boxes in the rightmost column the containment limit is still $\frac{1}{2}$.

\begin{theorem}\label{thm:sideshade}
Let
 $p=$\patt{0.5}{4}{2,1,4,3}[0/0,1/0,2/0,3/0,4/0,4/1,4/2,4/3,0/4,1/4,2/4,3/4,4/4],
 then
 \[
 s_n^+(p) =(n-2)!\left[\frac{(n-2)(n-3)}{2}-\sum_{k=2}^{n-2}\frac{(n-k-1)}{k}\right]
 \;\;\;\;\;\;\;\;\;\;\text{and}\;\;\;\;\;\;\;\;\;\;
 \lim_{n\to\infty}\frac{s_n^+(p)}{n!} = \frac12.
 \]
\end{theorem}

\begin{proof}
Note that if $p$ occurs in a permutation $\sigma\in S_n$, the shading tells us that there can be nothing above $4$,
 nothing below $1$, and nothing to the right of $3$. This means that a realisation of $p$ in $\sigma$ is a subword of 
 the form $a1n\sigma_n$, subject to the restriction $1<a<\sigma_n<n$. In other words, the pattern $p$ occurs in 
 $\sigma$ if and only if the letter $1$ occurs in $\sigma$ before the letter~$n$ and some letter occurring before 
 the letter $1$ is smaller than the last letter in $\sigma$ (which is not allowed to be $n$). 
 To calculate $s_n^+(p)$, we can count such permutations.

Let $1<k<\ell<n$ and $1<q<n$. There are $(n-3)!$ permutations $\sigma$ satisfying $\sigma_k=1,\sigma_\ell=n$ and $\sigma_n=q$.
 A permutation of this kind will contain the mesh pattern if and only if a letter smaller than $q$ occurs in one of the first $k-1$ places. 
 As usual, it is easier to count which permutations do not satisfy this condition: namely those whose first $k-1$ letters are all larger than $q$. 
 Note that this cannot happen if $k-1>n-q-1$, as there are only $n-q-1$ letters between $q$ and $n$. If on the other hand $k-1\leq n-q-1$, 
 we have $\binom{n-q-1}{k-1}$ possible choices of letters to put in the first $k-1$ places, each of which we can permute in $(k-1)!$ possible ways. 
 The remaining $(n-k-2)$ letters can be freely permuted, giving us $(n-k-2)!$ choices. Therefore, among the $(n-3)!$ permutations satisfying the restrictions above,
 \[
 \binom{n-q-1}{k-1}(k-1)!(n-k-2)!=(n-3)!\frac{\binom{n-q-1}{k-1}}{\binom{n-3}{k-1}}
 \]
 of them do not contain the mesh pattern. We subtract this number from $(n-3)!$ to obtain the number of those that do contain it:
 \[
 (n-3)!\left(1-\frac{\binom{n-q-1}{k-1}}{\binom{n-3}{k-1}}\right)
 \]
 To obtain the total number of such permutations in $S_n$ we now sum over all $k,\ell$ and $q$:
 \[
 s_n^+(p) = \sum_{q=2}^{n-1}\sum_{k=2}^{n-2}\sum_{\ell=k+1}^{n-1}(n-3)!\left(1-\frac{\binom{n-q-1}{k-1}}{\binom{n-3}{k-1}}\right).
 \]
 Since $\ell$ does not appear in the sum, we can replace the corresponding summation symbol by a multiplicative factor, giving the formula:
 \begin{align*}
  s_n^+(p) &= \sum_{q=2}^{n-1}\sum_{k=2}^{n-2}(n-k-1)(n-3)!\left(1-\frac{\binom{n-q-1}{k-1}}{\binom{n-3}{k-1}}\right)\\
  &=(n-3)!\sum_{k=2}^{n-2}\left[(n-2)(n-k-1)-\frac{(n-k-1)}{\binom{n-3}{k-1}}\sum_{q=2}^{n-1}\binom{n-q-1}{k-1}\right]\\
  &=(n-3)!\left[\frac{(n-2)^2(n-3)}{2}-\sum_{k=2}^{n-2}\frac{(n-k-1)}{\binom{n-3}{k-1}}\binom{n-2}{k}\right]\\
  &=(n-3)!\left[\frac{(n-2)^2(n-3)}{2}-\sum_{k=2}^{n-2}\frac{(n-k-1)(n-2)}{k}\right]\\
  &=(n-2)!\left[\frac{(n-2)(n-3)}{2}-\sum_{k=2}^{n-2}\frac{(n-k-1)}{k}\right]\\
 \end{align*}

Next, we calculate the containment limit.
 By Theorem~\ref{topbottom} and
 Lemma~\ref{lem:subshade} we get an upper bound of $\displaystyle\lim_{n\to\infty}\frac{s_n^+(p)}{n!} \le \frac12$.
 The lower bound is given by:
 \begin{align*}
 \lim_{n\to\infty}\frac{s_n^+(p)}{n!} &= \lim_{n\to\infty}\frac{1}{n(n-1)}\left[\frac{(n-2)(n-3)}{2}-\sum_{k=2}^{n-2}\frac{(n-k-1)}{k}\right] \\
          &= \lim_{n\to\infty}\frac{1}{n(n-1)}\left[\frac{(n-2)(n-3)}{2}-(n-1)(H_{n-2}-1)+(n-3)\right]
        =\frac12
 \end{align*}
 Combining the two bounds implies $\lim_{n\to\infty}\frac{s_n^+(p)}{n!}=\frac12$.
 \end{proof}

We cannot apply the Shading Lemma to Theorem~\ref{thm:sideshade}, since the shaded boxes on the right hand side
 cause the conditions of the Lemma to no longer be satisfied. But we do get the following corollary from trivial
 symmetries.

\begin{corollary}\label{cor:sideshade}
If $p$ is any of the following mesh patterns:
 \begin{center}
  \patt{0.25}{4}{2,1,4,3}[0/0,0/1,0/2,0/3,0/4,1/4,2/4,3/4,4/0,4/1,4/2,4/3,4/4]
  \patt{0.25}{4}{2,1,4,3}[0/0,0/1,0/2,0/3,0/4,1/0,2/0,3/0,4/0,4/1,4/2,4/3,4/4]
  \patt{0.25}{4}{3,4,1,2}[0/0,0/1,0/2,0/3,0/4,1/0,2/0,3/0,4/0,4/1,4/2,4/3,4/4]
  \patt{0.25}{4}{3,4,1,2}[0/0,0/1,0/2,0/3,0/4,1/4,2/4,3/4,4/0,4/1,4/2,4/3,4/4]
  \patt{0.25}{4}{2,1,4,3}[0/0,0/4,1/0,1/4,2/0,2/4,3/0,3/4,4/0,4/1,4/2,4/3,4/4]
  \patt{0.25}{4}{2,1,4,3}[0/0,0/1,0/2,0/3,0/4,1/0,1/4,2/0,2/4,3/0,3/4,4/0,4/4]
  \patt{0.25}{4}{3,4,1,2}[0/0,0/1,0/2,0/3,0/4,1/0,1/4,2/0,2/4,3/0,3/4,4/0,4/4]
  \patt{0.25}{4}{3,4,1,2}[0/0,0/4,1/0,1/4,2/0,2/4,3/0,3/4,4/0,4/1,4/2,4/3,4/4],
 \end{center}
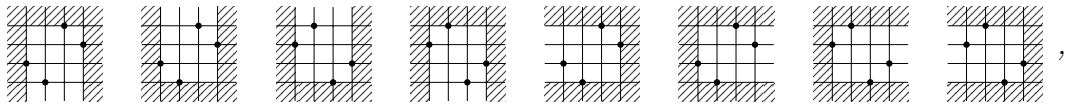
 \[
 \text{then}\;\;\;\;\;\
 s_n^+(p) = \sum_{q=2}^{n-1}\sum_{k=2}^{n-2}(n-k-1)(n-3)!\left(1-\frac{\binom{n-q-1}{k-1}}{\binom{n-3}{k-1}}\right)
 \;\;\;\;\;\;\;\text{and}\;\;\;\;\;\;\;
 \lim_{n\to\infty}\frac{s_n^+(p)}{n!} = \frac12.
 \]
\end{corollary}

\begin{theorem}\label{thm:nocorner}
Let $p=(2143,R)$ be the pattern with $R=([3]\times 0)\cup([4]\times 4)$, as in the picture
 \begin{center}
  \patt{0.5}{4}{2,1,4,3}[0/0,1/0,2/0,3/0,0/4,1/4,2/4,3/4,4/4]
 \end{center}
 Then,
 \[
 s_n^+(p) = \sum_{q=1}^{n-3}\sum_{k=1}^{n-q-2}\sum_{\ell=1}^{n-q-1-k}\frac{(n-2)!}{\binom{n-2}{q-1}}\binom{n-k-\ell-2}{q-1}\left(1-\binom{k+\ell}{k}^{-1}\right).
 \]
 Furthermore,
 \[
 \lim_{n\to\infty}\frac{s_n^+(p)}{n!} = 1.
 \]
\end{theorem}

\begin{proof}
Note that if $p$ occurs in a permutation $\sigma\in S_n$, then a realisation of $p$ in~$\sigma$ 
 is a subword of the form $aqnb$, subject to the restriction $q<a<b<n$ such that every letter in 
 $\sigma$ smaller than $q$ appears after $b$. In particular, $n$ cannot be in the first place. 
 To calculate~$s_n^+(p)$ we count such permutations.

Assume $n$ is not in the first position and let $q$ be the smallest letter appearing before~$n$. 
 The permutation $\sigma$ contains a unique subword $\omega=qn\pi_1\ldots\pi_{q-1}$ where $\pi_i<q$ for all $i$. 
 Fixing $q<n-2$, there are $(q-1)!$ choices for this subword and
 the rest of the letters of $\sigma$ can be chosen freely, in $(n-q-1)!$ different ways. 
 To obtain a permutation $\sigma$ from the subword $\omega$, we have to insert $n-q-1$ letters. 
 Let~${\psi=(\psi_0,\dots,\psi_{q+1})}$ be the number of letters we insert into each of the parts of $\sigma$ separated by the letters of $\omega$.

For a fixed $\pi$ and a fixed partition $\psi$, there are $(n-q-1)!$ possible permutations of letters to be inserted,
 and we now count how many of these do not contain $p$. These are exactly the permutations where all the $\psi_0$ 
 letters before $q$ are larger than all the $\psi_2$ letters between~$n$ and $\pi_1$, and the remaining letters 
 can be chosen and permuted arbitrarily. First we pick the remaining letters, of which there are $n-q-\psi_0-\psi_2-1$ 
 to be picked from~$n-q-1$ total letters, and permuted in any way. Next we take the $\psi_0$ largest unselected letters 
 and insert these before $q$, in any order, and the final $\psi_2$ letters are inserted between $n$ and $\pi_1$, in any order. This gives us
 \[
 \binom{n-q-1}{n-q-\psi_0-\psi_2-1}(n-q-\psi_0-\psi_2-1)!\psi_0!\psi_2!=\frac{(n-q-1)!}{\binom{\psi_0+\psi_2}{\psi_0}}
 \]
 permutations avoiding the mesh pattern for a fixed $\pi$ and $\psi$, and subtracting this from $(n-q-1)!$ gives the number of permutations containing $p$.

Summing over all possible choices of $\pi$ and $\psi$, we see that there are
 \begin{align*}
 a_{q,n}&\coloneqq(q-1)!\sum_{\psi} (n-q-1)!\left(1-\binom{\psi_0+\psi_2}{\psi_0}^{-1}\right)\\
 &=(q-1)!\sum_{\psi_0=1}^{n-q-2}\sum_{\psi_2=1}^{n-q-\psi_0-1}\sum_{\psi_1,\psi_3,\dots,\psi_{q+1}}(n-q-1)!\left(1-\binom{\psi_0+\psi_2}{\psi_0}^{-1}\right).
 \end{align*}
 permutations containing the pattern $p$ such that $q$ is the smallest letter preceding~$n$.

Note that the innermost sum does not depend on the value of either of the parameters we are summing over, 
 so it is the same as multiplying by the number of ordered partitions of $n-q-\psi_0-\psi_2-1$ into $q$ nonnegative parts. So the sum is equal to:
 \[
 a_{q,n}=(q-1)!\sum_{\psi_0=1}^{n-q-2}\sum_{\psi_2=1}^{n-q-\psi_0-1}\binom{n-\psi_0-\psi_2-2}{q-1}(n-q-1)!\left(1-\binom{\psi_0+\psi_2}{\psi_0}^{-1}\right).
 \]
 Summing over all possible $q$ (and replacing $\psi_0$ and $\psi_2$ by $k$ and $\ell$, respectively) gives the desired formula for $s_n^+(p)$.

Next we compute the asymptotic behaviour of $s_n^+(p)/n!$. First note that $s_n^+(p)\le n!$ as there are at most $n!$ permutations of length $n$,
 so $\frac{s_n^+(p)}{n!}\le \frac{n!}{n!}=1$. Now consider the lower bound.
 
 \begin{align}
   s_n^+(p) &= \sum_{q=1}^{n-3}\sum_{k=1}^{n-q-2}\sum_{\ell=1}^{n-q-1-k}\frac{(n-2)!}{\binom{n-2}{q-1}}\binom{n-k-\ell-2}{q-1}\left(1-\binom{k+\ell}{k}^{-1}\right)\label{nc1}\\
   &\ge \sum_{q=1}^{n-3}\frac{(n-2)!}{\binom{n-2}{q-1}}\sum_{k=1}^{n-q-2}\sum_{\ell=1}^{n-q-1-k}\binom{n-k-\ell-2}{q-1}\frac{k}{k+1}\label{nc2}\\
   &= \sum_{q=1}^{n-3}\frac{(n-2)!}{\binom{n-2}{q-1}}\sum_{k=1}^{n-q-2}\frac{k}{k+1}\binom{n-k-2}{q}\label{nc3}\\
   &= \sum_{q=1}^{n-3}\frac{(n-2)!}{\binom{n-2}{q-1}}\left[\binom{n-2}{q+1}-\sum_{k=1}^{n-q-2}\frac{1}{k+1}\binom{n-k-2}{q}\right]\label{nc4}\\
   &\ge \sum_{q=1}^{n-3}\frac{(n-2)!}{\binom{n-2}{q-1}}\left[\binom{n-2}{q+1}-\binom{n-3}{q}\sum_{k=1}^{n-q-2}\frac{1}{k+1}\right]\label{nc5}\\
   &\ge \sum_{q=1}^{n-3}\frac{(n-2)!}{\binom{n-2}{q-1}}\left[\binom{n-2}{q+1}-\binom{n-3}{q}\ln(n-q-1)\right]\label{nc6}\\
   &\ge   (n-2)!\left[\sum_{q=1}^{n-3}\frac{(n-q-1)(n-q-2)}{q(q+1)}-\frac{\ln(n-2)}{n-2}\sum_{q=1}^{n-3}\frac{(n-q-1)(n-q-2)}{q}\right]\label{nc7}\\
   &=   (n-2)!\left[(n+1)(n-2)-2(n-1)H_{n-2}\right]\label{nc8}\\&\qquad\qquad-(n-3)!\ln(n-2)\left[(n-2)(n-1)H_{n-3}-\frac{(n-3)(3n-4)}{2}\right]\nonumber\\
   &=   (n-2)!(n+1)(n-2)-2(n-1)!H_{n-2}-\ln(n-2)(n-1)!H_{n-3}\nonumber\\&\qquad\qquad+\frac{(n-3)!(n-3)(3n-4)\ln(n-2)}{2}\label{nc9}
 \end{align}
 where $H_n$ is the $n$-th Harmonic number.
 To get \eqref{nc2} we set $\binom{k+\ell}{k}$ to $\binom{k+1}{k}$,
 \eqref{nc3} is given by the identity $\sum_{a=b}^c\binom{a}{b}=\binom{c+1}{b+1}$,
 \eqref{nc4} we get by using $\frac{k}{k+1}=1-\frac{1}{k+1}$ and then sum of binomial coefficients identity again,
 \eqref{nc5} is obtained by setting the second $k$ to $1$,~\eqref{nc6} we get from the inequality $\sum_{k=1}^{n}\frac{1}{k}=H_n\le ln(n)+1$,
 \eqref{nc7} is given by dividing the binomial coefficients by~$\binom{n-2}{q-1}$,
 \eqref{nc8} we get by multiplying out the polynomials and simplifying the sums,
 and \eqref{nc9} is simply expanding the brackets.

Therefore, we get the following lower bound, which combined with the upper bound completes the proof:
 \begin{align*}
   \lim_{n\to\infty} \frac{s_n^+(p)}{n!}
            \ge\lim_{n\to\infty} \biggl[\frac{(n+1)(n-2)}{n(n-1)}&-\frac{2H_{n-2}}{n}-\frac{\ln(n-2)H_{n-3}}{n}\\&+\frac{(n-3)(3n-4)\ln(n-2)}{2n(n-1)(n-2)}\biggr]=1.
 \end{align*}
\end{proof}

Combining Lemma~\ref{lem:subshade} and Theorem~\ref{thm:nocorner} gives the following corollary.

\begin{corollary}
Let $p=(2143,R)$ where $R\subseteq([3]\times 0)\cup([4]\times 4)$ then,
 \[
 \lim_{n\to\infty}\frac{s_n^+(p)}{n!} = 1.
 \]
\end{corollary}

\begin{example}
Let $p=$
 \patt{0.5}{4}{2,1,4,3}[0/0,4/4],
 then
 $\displaystyle
 \lim_{n\to\infty}\frac{s_n^+(p)}{n!} = 1.
 $
\end{example}

We can apply the Shading Lemma to get the following corollary.

\begin{corollary}
 Suppose $p$ is any of the following mesh patterns
 \begin{align*}
  &\patt{0.25}{4}{2,1,4,3}[0/0,1/0,2/0,3/0,0/4,1/4,2/4,3/4,4/4,0/1]
  \patt{0.25}{4}{2,1,4,3}[0/0,1/0,2/0,3/0,0/4,1/4,2/4,3/4,4/4,0/2]
  \patt{0.25}{4}{2,1,4,3}[0/0,1/0,2/0,3/0,0/4,1/4,2/4,3/4,4/4,1/1]
  \patt{0.25}{4}{2,1,4,3}[0/0,1/0,2/0,3/0,0/4,1/4,2/4,3/4,4/4,1/2]
  \patt{0.25}{4}{2,1,4,3}[0/0,1/0,2/0,3/0,0/4,1/4,2/4,3/4,4/4,0/1,0/2]
  \patt{0.25}{4}{2,1,4,3}[0/0,1/0,2/0,3/0,0/4,1/4,2/4,3/4,4/4,1/2,1/1]
  \patt{0.25}{4}{2,1,4,3}[0/0,1/0,2/0,3/0,0/4,1/4,2/4,3/4,4/4,0/1,1/1]
  \patt{0.25}{4}{2,1,4,3}[0/0,1/0,2/0,3/0,0/4,1/4,2/4,3/4,4/4,1/2,0/2]\\
  &\patt{0.25}{4}{2,1,4,3}[0/0,1/0,2/0,3/0,0/4,1/4,2/4,3/4,4/4,4/2]
  \patt{0.25}{4}{2,1,4,3}[0/0,1/0,2/0,3/0,0/4,1/4,2/4,3/4,4/4,4/3]
  \patt{0.25}{4}{2,1,4,3}[0/0,1/0,2/0,3/0,0/4,1/4,2/4,3/4,4/4,3/3]
  \patt{0.25}{4}{2,1,4,3}[0/0,1/0,2/0,3/0,0/4,1/4,2/4,3/4,4/4,3/2]
  \patt{0.25}{4}{2,1,4,3}[0/0,1/0,2/0,3/0,0/4,1/4,2/4,3/4,4/4,3/3,3/2]
  \patt{0.25}{4}{2,1,4,3}[0/0,1/0,2/0,3/0,0/4,1/4,2/4,3/4,4/4,4/3,4/2]
  \patt{0.25}{4}{2,1,4,3}[0/0,1/0,2/0,3/0,0/4,1/4,2/4,3/4,4/4,3/3,4/3]
  \patt{0.25}{4}{2,1,4,3}[0/0,1/0,2/0,3/0,0/4,1/4,2/4,3/4,4/4,3/2,4/2]
 \end{align*}
 or obtained by the union of the shaded boxes from a mesh pattern on the top row and a mesh pattern on the bottom row. Then
 \[
 s_n^+(p) = \sum_{q=1}^{n-3}\sum_{k=1}^{n-q-2}\sum_{\ell=1}^{n-q-1-k}\frac{(n-2)!}{\binom{n-2}{q-1}}\binom{n-k-\ell-2}{q-1}\left(1-\binom{k+\ell}{k}^{-1}\right).
 \]
 and
 \[
 \lim_{n\to\infty}\frac{s_n^+(p)}{n!} = 1.
 \]
 Moreover, if $p$ obtained by taking a subset of the shadings of any of the above mesh patterns then the containment limit is also $1$.
\end{corollary}

\section{Questions and Conjectures}\label{sec:5}
We have presented some results on the asymptotics of mesh pattern containment, but many questions remain. 
 In this section we finish with some open questions and conjectures.

Lemma~\ref{lem:subshade}, Theorem~\ref{thm:sideshade}, and Theorem~\ref{thm:nocorner} imply that if
$$p=\patt{0.25}{4}{2,1,4,3}[0/0,1/0,2/0,3/0,4/1,4/2,4/3,0/4,1/4,2/4,3/4,4/4], \;\;\;\;\;\;\;\;\;\text{then} \;\;\;\;\;\;\;\;\;\frac12\le\lim_{n\to\infty}\frac{s_n^+(p)}{n!}\le1.$$
 Permutations that contain $p$ must be constructed by taking the skew-sum of a permutation that contains the mesh pattern in Theorem~\ref{thm:sideshade} along with any other permutation,
 this leads us to claim the following conjecture.
 
\begin{conjecture}
If $p=\patt{0.25}{4}{2,1,4,3}[0/0,1/0,2/0,3/0,4/1,4/2,4/3,0/4,1/4,2/4,3/4,4/4]$, then
 $ \displaystyle\lim_{n\to\infty}\frac{s_n^+(p)}{n!} = \frac12$.
\end{conjecture}

Many of the examples we consider have entire rows or columns shaded, so could actually be considered as bivincular patterns.
 So we wonder if a complete analysis of length $3$ or $4$ bivincular patterns is possible?
 Also, we have only considered $s_n^+(P)$, where $P$ is a single pattern. What can we determine for the containment limit
 when $P$ is a set of mesh patterns?

One motivation for this work was the application of mesh patterns to the homology of permutation complexes, such as those studied in \cite{CR20}.
 For example, we initially believed that each occurrence in~${\pi\in S_n}$ of the mesh pattern 
 \begin{center}\patt{0.3}{4}{2,1,4,3}[0/0,2/2,4/4]\end{center} 
 contributes to the homology of the permutation complex
 $X(12\dots n,\pi)$, which is topologically the suspension of the simplicial complex $Y(\pi)$ whose faces are the increasing subsequences of $\pi$.
 This was motivated by the observation that the topological realisation of the subcomplex corresponding to an occurrence of such a pattern in $Y(\pi)$
 is a circle which the shading prevents from being coned off by the addition of a single point. Unfortunately, it can be coned off by the addition 
 of two points, as demonstrated by taking the permutation $\pi=214635$ and the occurrence $2163$ of the pattern.
 Which leads us to pose the question: 
 \begin{question}
 Can occurrences of mesh patterns be used to compute the Betti numbers of permutation complexes?
 Or can we define a set of mesh patterns $P$ such that if $\pi$ avoids $P$ its permutation complex is contractible?
 \end{question}

\section*{Acknowledgements}
The \emph{Permuta} python package \cite{permuta} was used for some exploratory computations.


\begin{thebibliography}{BGMU19}

\bibitem[AAB{\etalchar{+}}]{permuta}
Ragnar~Pall Ardal, Arnar~Bjarni Arnarson, Christian Bean, Unnar~Freyr
  Erlendsson, {\'E}mile Nadeau, Jay Pantone, Tomas~Ken Shimomura-Magnusson, and
  Henning Ulfarsson.
\newblock Permuta.
\newblock \url{https://github.com/PermutaTriangle/Permuta}.

\bibitem[AKV13]{AKV13}
Sergey Avgustinovich, Sergey Kitaev, and Alexandr Valyuzhenich.
\newblock Avoidance of boxed mesh patterns on permutations.
\newblock {\em Discrete Applied Mathematics}, 161(1-2):43--51, 2013.

\bibitem[BC11]{Bra11}
Petter Br{\"a}nd{\'e}n and Anders Claesson.
\newblock Mesh patterns and the expansion of permutation statistics as sums of
  permutation patterns.
\newblock {\em Electron. J. Combin}, 18(2):P5, 2011.

\bibitem[BGMU19]{BGMU19}
Christian Bean, Bjarki Gudmundsson, Tomas~Ken Magnusson, and Henning Ulfarsson.
\newblock Algorithmic coincidence classification of mesh patterns.
\newblock {\em arXiv preprint arXiv:1910.08127}, 2019.

\bibitem[CLM20]{CR20}
Wojtek Chacholski, Ran Levi, and Roy Meshulam.
\newblock On the topology of complexes of injective words.
\newblock {\em Journal of Applied and Computational Topology}, 4(1):29--44,
  2020.

\bibitem[CTU15]{CTU15}
Anders Claesson, Bridget~Eileen Tenner, and Henning Ulfarsson.
\newblock {Coincidence among families of mesh patterns}.
\newblock {\em The Australasian Journal of Combinatorics}, 63:88--106, 2015.

\bibitem[Eul41]{Eul41}
Leonhard Euler.
\newblock Inventio summae cuiusque seriei ex dato termino generali.
\newblock {\em Commentarii academiae scientiarum Petropolitanae}, pages 9--22,
  1741.

\bibitem[HJS{\etalchar{+}}15]{HJSVU15}
{\'I}sak Hilmarsson, Ingibj{\"o}rg J{\'o}nsd{\'o}ttir, Steinunn
  Sigur{\dh}ard{\'o}ttir, L{\'\i}na Vi{\dh}arsd{\'o}ttir, and Henning
  Ulfarsson.
\newblock Wilf-classification of mesh patterns of short length.
\newblock {\em The Electronic Journal of Combinatorics}, 22(4):P4--13, 2015.

\bibitem[JKR15]{JKR15}
Miles Jones, Sergey Kitaev, and Jeffrey Remmel.
\newblock {Frame patterns in n-cycles}.
\newblock {\em Discrete Mathematics}, 338(7):1197--1215, July 2015.

\bibitem[Kit11]{Kit11}
Sergey Kitaev.
\newblock {\em Patterns in Permutations and Words}.
\newblock Monographs in Theoretical Computer Science. An EATCS Series.
  Springer, Heidelberg, 2011.

\bibitem[KZ19]{KZ19}
Sergey Kitaev and Philip~B Zhang.
\newblock Distributions of mesh patterns of short lengths.
\newblock {\em Advances in Applied Mathematics}, 110:1--32, 2019.

\bibitem[KZZ20]{KZZ20}
Sergey Kitaev, Philip~B Zhang, and Xutong Zhang.
\newblock Distributions of several infinite families of mesh patterns.
\newblock {\em Applied Mathematics and Computation}, 372:124984, 2020.

\bibitem[MT04]{MT04}
Adam Marcus and G{\'a}bor Tardos.
\newblock Excluded permutation matrices and the {S}tanley--{W}ilf conjecture.
\newblock {\em Journal of Combinatorial Theory, Series A}, 107(1):153--160,
  2004.

\bibitem[Ten13]{Ten13}
Bridget~Eileen Tenner.
\newblock Mesh patterns with superfluous mesh.
\newblock {\em Advances in Applied Mathematics}, 51(5):606--618, 2013.

\end{thebibliography}
\newcommand{\etalchar}[1]{$^{#1}$}
 \newcommand{\noop}[1]{}

\end{document}